\documentclass[12pt, reqno]{amsart}
\usepackage{amssymb,amsthm,amsfonts,amsmath, pst-all}

\usepackage{hyperref}
\usepackage{mathrsfs}

\newcommand{\Z}{\mathbb Z}

\newcommand{\Prob}{\mathbb P}
\newcommand\Y{\mathbb Y}

\newcommand{\la}{\lambda}

\newcommand{\si}{\sigma}
\newcommand\Si{\Sigma}
\newcommand\epsi{\varepsilon}

\newcommand{\Sym}{\mathfrak S}

\newcommand{\Q}{\mathcal Q}
\newcommand\G{\mathcal G}

\newcommand\adm{{\operatorname{adm}}}
\newcommand\bal{{\operatorname{bal}}}
\newcommand\fin{{\operatorname{fin}}}
\newcommand\inv{{\operatorname{inv}}}
\newcommand\rk{{\operatorname{rk}}}

\newcommand\Ord{{\operatorname{Ord}}}
\newcommand\sgn{{\operatorname{sgn}}}
\newcommand\const{{\operatorname{const}}}

\newtheorem{theorem}{Theorem}[section]
\newtheorem{proposition}[theorem]{Proposition}
\newtheorem{corollary}[theorem]{Corollary}
\newtheorem{lemma}[theorem]{Lemma}

\theoremstyle{definition}
\newtheorem{definition}[theorem]{Definition}
\newtheorem{remark}[theorem]{Remark}
\newtheorem{example}[theorem]{Example}

\begin{document}
\title[Extension of the Mallows model for random permutations]
{The two-sided infinite extension of the Mallows model for random permutations}
\author{Alexander Gnedin}
\address{Department of Mathematics, Utrecht University, the Netherlands}
\email{A.V.Gnedin@uu.nl}

\author{Grigori Olshanski}
\address{Institute for Information Transmission Problems, Moscow, Russia;
\newline\indent and Independent University of Moscow, Russia} \email{olsh2007@gmail.com}

\date{}

\thanks{G.~O. was supported by the RFBR-CNRS grant 10-01-93114
and the project SFB 701 of Bielefeld University.}

\begin{abstract}
\noindent We introduce a probability distribution $\Q$ on the infinite group
$\Sym_\Z$ of permutations of the set of integers $\Z$. Distribution $\Q$ is a
natural extension of the Mallows distribution on the finite symmetric group. A
one-sided infinite counterpart of $\Q$, supported by the group of permutations
of $\mathbb N$, was studied previously in our paper [Gnedin and Olshanski, Ann.
Prob. 38 (2010), 2103-2135; arXiv:0907.3275]. We analyze various features of
$\Q$ such as its symmetries, the support, and the marginal distributions.
\end{abstract}

\maketitle

\tableofcontents

\section{Introduction}\label{section1}

Let ${\mathfrak S}_n$ be the group of permutations of $\{1,\dots,n\}$. A random
permutation $\Sigma_n$ with the probability distribution
\begin{equation}\label{MalFin}
\Prob(\Sigma_n=\sigma)=c^{-1}_nq^{\inv(\sigma)}, \qquad 0<q<1,
\end{equation}
is one of the Mallows  models on $\mathfrak S_n$, see \cite{DR}. Here $\sigma$
ranges over $\Sym_n$,
$$
\inv(\sigma)=\#\{(i,j): 1\leq i< j\leq n,~\sigma(j)>\sigma(i)\}
$$
is the number of inversions, and
$$
c_n=[n!]_q=\prod_{i=1}^n\frac{1-q^i}{1-q}
$$
is the normalizing constant (a $q$-analog of the factorial $n!$). Distribution
\eqref{MalFin} is qualitatively different from the uniform in that it favors
the order: the probability is maximal at the identity permutation
$\sigma(i)=i$, and falls off exponentially as the number of inversions
increases.

The Mallows model  was introduced in connection with ranking problems in
statistics (see \cite{Mallows} and a more general definition in \cite[p. 104,
Example 3]{Di88}) and recently it appeared in connection with random sorting
algorithms and trees \cite{Benjamini}, \cite{DR}, \cite{Evans}. The
distributions of displacements $\Sigma_n(i)-i$ are tight as $n\to\infty$ for
each fixed $i$, which suggests understanding $\Sigma_n$ as a random function
with linear trend. To compare, the order of displacements in the uniform
permutation is $O(n)$. See \cite{Muller} and references therein for some
large-$n$ properties of the Mallows model when $q=q_n$ varies in such a way
that $(1-q_n)n\to{\rm const}$, so $\Sigma_n$ approaches the uniform permutation
in this regime.

In our previous paper \cite{Qexch} we observed that \eqref{MalFin} is the
unique distribution on ${\mathfrak S}_n$ with the property of
$q$-exchangeability, which means that by swapping the values in any two
adjacent positions $i$ and $i+1$ the probability of permutation is multiplied
by $q^{{\rm sgn}(\sigma(i+1)-\sigma(i))}$ where $\sgn(x)=\pm1$ according to the
sign of $x\ne0$. In the course of generalizing this property to arbitrary
infinite real-valued sequences we were lead  to introducing a random
permutation (bijection) $\Si_+:\Z_+\to\Z_+$ which extends \eqref{MalFin} to the
case ``$n=\infty$'', in the sense that, in suitable coordinates, the Mallows
measures appear as  consistent $n$-dimensional marginal distributions of
${\Si}_+$. One of the uses of $\Si_+$ is that many large-$n$ features of
$\Sigma_n$ can be recognized as properties of $\Sigma_+$.

In this paper we  introduce a further {\it two-sided infinite\/} extension of
the Mallows model, which is a random $q$-exchangeable bijection
$\Sigma:\Z\to\Z$. Permutations $\Sigma_+$ and $\Sigma$ share many common
features, among which are the invariance of the distribution under passing to
the inverse permutation, and the related property of quasi-invariance under
swapping positions of two adjacent  values  $j$ and $j+1$. Both  $\Sigma_+$ and
$\Sigma$ can be constructed from independent copies of the same geometric
random variable, but $\Sigma$ has more symmetries. A peculiar feature of the
two-sided counterpart is that the process of displacements $(\Sigma(i)-i,
~i\in\Z)$ is stationary. We shall describe the support of $\Sigma$ and derive
formulas for the joint distribution of the displacements in terms of some
series of the $q$-hypergeometric type.

From a general perspective, \eqref{MalFin} is a conditionally uniform
distribution on ${\mathfrak S}_n$ obtained as deformation of the uniform
distribution by exponential tilting of  distribution of some statistic of
permutations. Replacing $\inv(\sigma)$ in \eqref{MalFin} by the number of
cycles of $\sigma$ will yield the familiar Ewens distribution \cite{ABT}, which
is also an instance of the general Mallows model \cite[Section 4.4]{DH}. See
\cite{GOdescent}, \cite{GnedinDiscM} for other choices of the statistic.
Remarkably, the extended infinite counterparts of \eqref{MalFin} live on the
group of permutations, albeit these are no longer {\it finitary} permutations
that move finitely many integers. To compare, the (one-sided or two-sided)
infinite extension  of the uniform and of Ewens' measures on $\Sym_n$ are not
supported by the space of permutations (see \cite{KOV1}, \cite{KOV2},
\cite{Olsh} for a realization of extended Ewens' measures in the space of
virtual permutations). The extended Mallows and Ewens measures are
quasi-invariant with respect to left and right shifts by finitary permutations,
and both may be viewed as substitutes of the nonexisting finite Haar measure on
the  group of finitary permutations. The nice quasi-invariance properties of
the measures make it possible to construct families of unitary representations
of the group of finitary permutations, although it is still to be explored if
the extended Mallows measures may play in the harmonic analysis the role
similar to that of the  Ewens measure on virtual permutations (see \cite{KOV1},
\cite{KOV2}, \cite{Olsh}).

\section{Preliminaries on infinite permutations}\label{section2}

This section is aimed to introduce various classes of permutations, and to
distinguish the permutations considered as support of the to-be-constructed
measures from the permutations considered as transformations acting on this
support.

We use the standard notation $\Z$ for the set of integers. For $a\in\Z$ we set
$\Z_{<a}:=\{i\in\Z: i<a\}$. Likewise, we define the subsets $\Z_{\le a}$,
$\Z_{>a}$, and $\Z_{\ge a}$. A nonstandard convention of this paper is that
$$
\Z_-:=\Z_{\le0}=\{\dots,-1,0\}, \qquad \Z_+:=\Z_{\ge1}=\{1,2,\dots\}.
$$

Under {\it permutation of\/ $\Z$} we shall understand an arbitrary bijection
$\sigma:\Z\to\Z$. Let $\mathfrak S$ denote the group of all permutations of
$\Z$. We associate with $\si\in\Sym$ an infinite 0\,-1 matrix $A=A(\si)$ of
format $\Z\times\Z$ such that the $(i,j)$th entry of $A(\si)$ is ${\bf
1}(\si(j)=i)$. Here and throughout ${\bf 1}(\cdots)$ equals $1$ if the
condition $\cdots$ is true, and equals $0$ otherwise. Observe that the group
operation on permutations agrees with the matrix multiplication, that is
$A(\si\tau)=A(\si)A(\tau)$. Write $A=A(\si)$ as a $2\times2$ block matrix
$$
A=\begin{bmatrix} A_{--} & A_{-+}\\ A_{+-} & A_{++} \end{bmatrix}
$$
according to the splitting $\Z=\Z_-\sqcup\Z_+$. For matrix $B$   let $\rk(B)$
be the rank of $B$, which is equal to the number of 1's in $B$ if $B$ is a
submatrix of $A(\sigma)$. We call $\si$ {\it admissible\/} if both
$\rk(A_{-+})$ and $\rk(A_{+-})$ are finite. The set of admissible permutations
will be denoted $\Sym^\adm\subset\Sym$.

\begin{remark}\label{rem2}
There is a similarity between our  definition of admissibility and the concept
of {\it restricted\/} matrix for infinite-dimensional classical matrix groups,
as found e.g. in \cite[Definition 6.2.1]{PS}.
\end{remark}

For  $\sigma\in\Sym^\adm$ we define the {\it balance} as
$$
b(\si):=\rk(A_{-+})-\rk(A_{+-}).
$$
It is readily checked that the same value  $b(\si)$ appears if instead of the
splitting $\Z=\Z_-\sqcup\Z_+$ the difference of ranks is computed with respect
to  any other splitting of the form $\Z=\Z_{\le a}\sqcup\Z_{>a}$ with $a\in\Z$.

\begin{example} For a {\it shift permutation}
$$
s^{(b)}:i\mapsto i-b, \qquad b\in\Z,
$$
we have
$$
\rk(A_{-+})=\begin{cases} n, & b>0\\ 0, & b\le0 \end{cases} \quad \textrm{and}
\quad  \rk(A_{+-})=\begin{cases} |b|, & b<0\\ 0, & b\ge0 \end{cases}
$$
It follows that $s^{(b)}$ is admissible and has balance $b$.
\end{example}

\medskip

For $I\subset \Z$ let $\Sym_I$ denote the set of permutations which satisfy
$\sigma(j)=j$ for $j\neq I$. The union
$$
\Sym^\fin:=\cup_{\{I:\#I<\infty\}}\Sym_I
$$
is the group of {\it finitary} permutations $\sigma$ which move only finitely
many integers, i.e. fulfill $\sigma(j)=j$ for $|j|$ sufficiently large. The
group $\Sym^\fin$ is countable, and is generated by the elementary
transpositions $\si_{i,i+1}$ that swap two adjacent integers $i,i+1\in\Z$. Note
that
$$
b(\si'\si)=b(\si\si')=b(\si), \qquad \si\in\Sym^\adm, \quad \si'\in\Sym^\fin,
$$
which is easy to check first  for the elementary transpositions and then by
induction for all $\si'\in \Sym^\fin$. It follows that every finitary
permutation is admissible and has balance 0.

\begin{proposition}\label{prop1}
{\rm(i)} $\Sym^\adm$ is a subgroup in $\Sym$.

{\rm(ii)} The balance is an additive character on $\Sym^\adm$, i.e.,
$$
b(\si\tau)=b(\si)+b(\tau), \qquad \si,\,\tau\in\Sym^\adm.
$$
\end{proposition}

\begin{proof}
(i) This is obvious, for the matrix multiplication preserves finiteness of the
ranks for the off-diagonal blocks.

(ii) It is readily checked that for each $\si\in\Sym^\adm$,
\begin{equation}\label{eq5}
b(s^{(n)}\si)=b(\si s^{(n)})=b(\si)+n, \qquad n\in\Z.
\end{equation}
Using this, one can show that $\si$ can be written in the form
$$
\si=s^{(k)}\si'\si'', \qquad k=b(\si),
$$
where $\si'\in\Sym^\fin$ and $A(\si'')$ is a block-diagonal matrix.
Likewise, $\tau$ is representable as the product
$$
\tau=\tau''\tau's^{(l)}, \qquad l=b(\tau),
$$
where $\tau'\in\Sym^\fin$ and $A(\tau'')$ is a block-diagonal matrix.
It follows that the permutation
$$
\si\tau=s^{(k)}\si'\si''\tau''\tau's^{(l)}
$$
has the same balance as $s^{(k)}s^{(l)}=s^{(k+l)}$. Thus,
$b(\si\tau)=k+l=b(\si)+b(\tau)$, which concludes the proof.
\end{proof}

We call a permutation {\it balanced\/} if it is admissible and has balance 0.
We shall denote the set of balanced permutation $\Sym^\bal$. By Proposition
\ref{prop1}, $\Sym^\bal$ is a normal subgroup in $\Sym^\adm$, and the quotient
group $\Sym^\adm/\Sym^\bal$ is isomorphic to $\Z$.

Every finitary permutation is balanced. An example of admissible permutation
which is not finitary is the shift $s^{(n)}$ with $n\ne0$. An example of a
balanced permutation which is not finitary  is the permutation which swaps each
even integer $2i$ with $2i+1$.

We conclude the section with a brief discussion of topologies on infinite
permutations. Let $E$ denote the coordinate Hilbert space $\ell^2(\Z)$ and let
$U(E)$ be the group of unitary operators in $E$. This is a metrizable
topological group with respect to the weak operator topology, which coincides
(on unitary operators only!) with the strong operator topology. The assignment
$\si\to A(\si)$ determines an embedding $\Sym\to U(E)$, and we endow the group
$\Sym$ with the weak topology inherited from $U(E)$. This weak topology on
$\Sym$ is non-discrete and totally disconnected. (Note that taking the norm
topology on $U(E)$ would produce on $\Sym$ the discrete topology.)

Equivalently, the weak topology on $\Sym$ can also be described as the topology
inherited from the compact product space $\{0,1\}^{\Z\times\Z}$, where the
embedding $\Sym\to \{0,1\}^{\Z\times\Z}$ is determined by the entries of
permutation matrices. Convergence $\sigma_n\to\sigma$ in the weak topology
means that for every finite submatrix of $A(\sigma_n)$ the entries stabilize
for sufficiently large $n$. An equivalent condition is that for any fixed
$i\in\Z$ one has $\si_n(i)=\si(i)$ for all $n$ large enough: this fact is a
special instance of the coincidence of the weak and the strong operator
topologies on $U(E)$.

Note that the group $\Sym^\fin$ is dense in $\Sym$.

The probability measures on $\Sym$ or on its subgroups to be introduced in the
sequel will be Borel with respect to the weak topology. Convergence of such
measures will be understood in the weak sense. Specifically, with each measure
$\mu$ and finite set $I\subset \Z$ we associate a discrete measure $\mu_I$,
which is the projection of $\mu$ on the finite space $\{0,1\}^{I\times
I}\subset \{0,1\}^{\Z\times\Z}$ of 0\,-1 matrices over $I$. Convergence of a
sequence measures $\mu^{(n)}$ to $\mu$ means convergence of the projections
$\mu^{(n)}_I$ to $\mu_I$, and the latter simply amounts to the natural
convergence of measures on finite sets.

Finally, note that the subgroup $\Sym^\adm\subset\Sym$ can be endowed with a
finer topology; it is defined by taking as a fundamental system of
neighborhoods of the identity element the following system of subgroups in
$\Sym^\bal$ indexed by arbitrary finite integer intervals $[a,b]\subset\Z$,
$a,b\in\Z$, $a\le b$:
$$
\Sym^{[a,b]}:=\{\si\in\Sym: \si(\Z_{<a})=\Z_{<a}, \quad \si(\Z_{>b})=\Z_{>b},
\quad \si(i)=i, \; i\in[a,b]\}.
$$
For some reasons related to Remark \ref{rem2} this finer topology on
$\Sym^\adm$ seems to be more natural than the weak topology, but for our
purposes the weak topology is enough; we will exploit it in Section
\ref{section7}.

\section{$q$-Exchangeability and the interlacing construction}\label{section3}

The Mallows measures \eqref{MalFin} and their extensions possess a fundamental
quasi-invariance property. Recall that $\si_{i,i+1}$ denotes the elementary
transposition swapping two adjacent indices $i,\,i+1\in\Sym$.

\begin{definition}\label{def2}
Fix $q>0$ and let $\mu$ be a measure on the group $\Sym$. Following
\cite{Qexch} we say that $\mu$ is {\it right $q$-exchangeable\/} if for every
$i\in\Z$, the pushforward $\mu_{i,i+1}$ of the measure $\mu$ under the
transformation $\si\to\si\si_{i,i+1}$ is equivalent to $\mu$ and the value of
the Radon-Nikod{\'y}m derivative $d\mu_{i,i+1}/d\mu$ at an arbitrary point
$\si\in\Sym$ equals $q^{\sgn(\si(i+1)-\si(i))}$. Likewise, we call $\mu$ {\it
left $q$-exchangeable\/} if the similar condition holds for transformations
$\si\to\si_{i,i+1}\si$, with the Radon-Nykod{\'y}m derivative
$q^{\sgn(\si^{-1}(i+1)-\si^{-1}(i))}$.
\end{definition}

Note that if $\mu$ is right $q$-exchangeable, then its pushforward $\mu'$ under
the transformation $\si\to\si^{-1}$ is left $q$-exchangeable, and vice versa.

Definition \ref{def2} is obviously extended to the setting where  $\Sym$ is
replaced by the group $\Sym_I$ of permutations of a finite or semi-finite
interval $I$ of the ordered set $\Z$.

If $I$ is finite, then $\Sym_I$ is isomorphic to the symmetric group $\Sym_n$
of degree $n=\#I$. Then it is readily seen that the notions of right and left
versions of $q$-exchangeability coincide and mean that $\mu(\si)$ is
proportional to $q^{\inv(\si)}$, as in \eqref{MalFin}. Therefore, if we
additionally require that $\mu$ is a probability measure, such a measure is
unique and coincides with \eqref{MalFin}, subject to the relabelling of the
elements of $I$ by increasing bijection with $\{1,\ldots,n\}$.

In \cite{Qexch} we proved the following result:

\begin{theorem}\label{thm1}
Assume $0<q<1$. On the group $\Sym_{\Z_+}$ the two notions of
$q$-exchangeability coincide, and there exists a unique probability measure
$\Q^+$, which is both right and left $q$-exchangeable.
\end{theorem}

(Our notation for $\Q^+$ in \cite{Qexch} was $\Q$, but now we reserve $\Q$ for
the two-sided extension.)

Obviously, the same result holds for permutations $\Sym_I$ of any semi-infinite
interval $I\subset\Z$ of the form $I=\Z_{\ge b}$ or $I=\Z_{\le a}$. Following
\cite{Qexch}, we call $\Q^+$ the {\it Mallows measure on\/ $\Sym_{\Z_+}$\/}.
The counterpart for $I=\Z_-$ will be called the {\it Mallows measure on\/
$\Sym_{\Z_-}$\/} and denoted  $\Q^-$.

Now we aim at proving an analog of Theorem \ref{thm1} for permutations of the
whole lattice $\Z$:

\begin{theorem}\label{thm2}
Assume $0<q<1$. On the group $\Sym^\adm\subset\Sym$ the two notions of
$q$-exchangeability coincide. There exists a unique probability measure $\Q$ on
$\Sym^\bal\subset\Sym$ which is both right and left $q$-exchangeable.
\end{theorem}

We call $\Q$ the {\it Mallows measure on\/ $\Sym^\bal$}. Before proceeding to
the proof we list some immediate corollaries and comments.

\begin{corollary}\label{cor1}
The measure $\Q$ is invariant under the inversion map $\si\to\si^{-1}$,
\end{corollary}

Indeed, the inversion turns the left $q$-exchangeability into the right
$q$-exchangeability and vice versa.

\begin{corollary}\label{cor2}
Let $\tau:\Z\to\Z$ stand for the reflection with respect to a given
half-integer $n+\frac12$. The measure $\Q$ is invariant under the conjugation
$\si\to\tau\si\tau^{-1}$.
\end{corollary}

Applying the shift transformations $\si\to s^{(b)}\si$ derives  from $\Q$ a family
$\{\Q^{(b)}\}_{b\in\Z}$ of $q$-exchangeable measures with pairwise disjoint
supports, which leave on the larger subgroup $\Sym^\adm\subset\Sym$.

\begin{corollary}\label{cor3}
Each $q$-exchangeable probability measure on $\Sym$ is a unique convex  mixture
of the measures $\Q^{(b)}$ over $b\in\Z$.
\end{corollary}

Thus, on the group $\Sym$ the uniqueness property is broken. The
reason is that the symmetry group $\operatorname{Aut}(\Z,<)$ of the ordered set
$(\Z,<)$ is the nontrivial group of shifts. However, the uniqueness is resurrected
if we  factorize the group $\Sym$ modulo the subgroup
$\operatorname{Aut}(\Z,<)$ (no matter, on the left or on the right). Finitary
permutations still act on this quotient space from  both sides, and there is
again a unique $q$-exchangeable probability measure which is
the push-forward of $\Q$.

The concept of $q$-exchangeability makes sense also for $q>1$,  but it is
easily reduced to $\bar{q}$-exchangeability with $\bar{q}=1/q\in (0,1)$ by
passing from $\sigma$ to permutation $i\to-\sigma(i)$.

\medskip

The plan of the proof of Theorem \ref{thm2} is the following.  We construct
$\Q$ from the Mallows measures on $\Z_+$ and $\Z_-$, and a random interlacing
pattern encoding how to fit the one-sided infinite extensions together. The
group $\Sym^\bal$ contains $\Sym_{\Z_+}\times\Sym_{\Z_-}$ as a subgroup, and
there is a natural bijection between the quotient set
$(\Sym_{\Z_+}\times\Sym_{\Z_-})\backslash \Sym^\bal$ and the set $\Y$ of Young
diagrams, which leads to a parametrization
\begin{equation}\label{eq1}
\Sym^\bal\ni\si\,\longleftrightarrow\,(\si^+,\si^-,\la)\in
\Sym_{\Z_+}\times\Sym_{\Z_-}\times\Y.
\end{equation}
In these coordinates we define $\Q$ as a product measure
\begin{equation}\label{eq2}
\Q:=\Q^+\otimes\Q^-\otimes {\mathcal P},
\end{equation}
where ${\mathcal P}$ is a probability measure on $\Y$ specified below in
\eqref{eq3}. A simple argument shows that $\Q$ is right $q$-exchangeable (Lemma
\ref{lemma1}). Next, we prove that $\Q$ is also left $q$-exchangeable, which is
somewhat more complicated (Lemma \ref{lemma2}). Note that these steps rely
heavily on the $q$-exchangeability property of the measures $\Q^\pm$
established in \cite{Qexch}. Finally, we verify the uniqueness claim (Lemma
\ref{lemma3} and Lemma \ref{lemma4}), which is again an easy exercise.

\begin{proof}[Proof of Theorem \ref{thm2}]
We proceed with the construction of $\Q$. Given a permutation
$\si\in\Sym^\bal$, we represent it by the two-sided infinite permutation word
$$
w=(\dots w_{-1}w_0 w_1 w_2 \dots)=(\dots\si(-1)\si(0)\si(1)\si(2)\dots),
$$
from which we derive a binary word
$$
\epsi=(\dots\epsi_{-1}\epsi_0\epsi_1\epsi_2\dots)\in\{-1,+1\}^\Z,
$$
where
$$
\epsi_i=\begin{cases} +1, & \si(i)\in\Z_+, \\ -1, & \si(i)\in\Z_-.\end{cases}
$$
The binary word $\epsi$ encodes the way in which positive and negative entries
of $w$ interlace. Obviously, $\epsi$ is a full invariant of the coset
$(\Sym_{\Z_+}\times\Sym_{\Z_-})\si$. The set $E$ of interlacing patterns that
stem from balanced permutations is characterized by the condition
\begin{equation}\label{eq4}
\#\{i\in\Z_-: \epsi_i=+1\}=\#\{i\in\Z_+: \epsi_i=-1\}<\infty.
\end{equation}

A well-known fact, often used in combinatorics, is that binary words $\epsi\in E$
can be conveniently encoded into Young diagrams $\la=(\la_1,\la_2,\dots)$ in
the following way. Split $\Z$ into the disjoint union of the sets of positions occupied by
$-1$'s and $+1$'s, respectively:
\begin{equation}\label{eq9}
\Z=\si^{-1}(\Z_-)\sqcup\si^{-1}(\Z_+)=\{\dots<j_{-2}<j_{-1}<j_0\}\sqcup\{i_1<i_2<i_3<\dots\}.
\end{equation}
Then $\la$ is defined by
\begin{equation}\label{eq13}
\la_1=1-i_1,\quad \la_2=2-i_2, \quad \la_3=3-i_3, \dots
\end{equation}
Here we identify Young diagrams and partitions, as in \cite{Ma}. Note also that
$$
\la'_1=j_0, \quad \la'_2=j_{-1}+1, \quad \la'_3=j_{-2}+2, \dots
$$
where $\la'=(\la'_1,\la'_2,\dots)$ is the transposed diagram; this is seen from
\cite[Chapter I, Equation (1.7)]{Ma}.

We define $\Q$ by \eqref{eq2}, taking for $\mathcal P$  the  probability
distribution on $\Y$ associated with Euler's partition generating function:
\begin{equation}\label{eq3}
{\mathcal P}(\la)=\const^{-1} q^{|\la|}, \qquad \la\in\Y, \quad
\const=\sum_{\la\in\Y}q^{|\la|}=\prod_{k=1}^\infty\left(\frac1{1-q^k}\right),
\end{equation}
where $|\la|=\la_1+\la_2+\dots$ is the number of boxes in the diagram $\la$.
Note that  $|\la|=|\la'|$ is  equal to the number of inversions in the
binary word $\epsi$, which is
$$
\#\{(i,j)\,:\, i<j, \, \epsi_i=+1,\,  \epsi_j=-1\}.
$$

Together with permutations $\si^\pm\in\Sym_{\Z_\pm}$ we will also consider the
corresponding permutation words $w^\pm$. Here
\begin{equation}\label{eq10}
w^+:=(\si^+(1)\si^+(2)\si^+(3)\dots)=(w_{i_1}w_{i_2}w_{i_3}\dots)
\end{equation}
is a permutation of $\Z_+$ while
\begin{equation}\label{eq11}
w^-:=(\dots\si^-(-2)\si^-(-1)\si^-(0))=(\dots w_{j_{-2}} w_{j_{-1}}w_{j_0})
\end{equation}
is a permutation of $\Z_-$.

To recover $w$ from the triple $(w^+,w^-,\epsi)$, one has to replace each $+1$
entry of the word $\epsi$  left-to-right with the successive entries of $w^+$,
and replace $-1$'s right-to-left with the entries of $w^-$.

\begin{lemma}\label{lemma1}
$\Q$ is right $q$-exchangeable.
\end{lemma}

\begin{proof}
Fix $i\in\Z$ and examine the behavior of $\Q$ under the right shift
$\si\to\si\si_{i,i+1}$. Encode $\si$ (or rather the corresponding word $w$) by
the triple $(w^+,w^-,\epsi)$. Depending on the $\epsi$-component,
there are three possible cases: $\epsi_i=\epsi_{i+1}=+1$,
$\epsi_i=\epsi_{i+1}=-1$, and $\epsi_i\ne\epsi_{i+1}$.

In the first case, both $w^-$ and $\epsi$ remain intact, and the transformation
$\si\to\si\si_{i,i+1}$ reduces to swapping two adjacent letters in $w^+$, whose
positions depend only on $\epsi$. Then the desired transformation property of
$\Q$ follows from the fact that measure $\Q^+$ is right $q$-exchangeable.

This kind of the argument also works in the second case, with appeal to the
similar property of $\Q^-$.

Finally, in the third case, the transformation affects only the binary word
$\epsi$ and amounts to swapping of two distinct adjacent letters.
This  changes the total number of inversions in $\epsi$ by $\pm1$. Then the desired
transformation property follows from the very definition of distribution $\mathcal P$.
\end{proof}

\begin{lemma}\label{lemma2}
$\Q$ is left $q$-exchangeable.
\end{lemma}

\begin{proof}
Let $W$ denote the set of all words $w$ corresponding to permutations
$\si\in\Sym^\bal$. As above, we represent $W$ as the direct product $W^+\times
W^-\times E$, where $W^\pm$ is the set of words $w^\pm$ corresponding to
permutations $\si^\pm\in\Sym_{\Z_\pm}$ and $E$ is the set of binary words
satisfying \eqref{eq4}.

A qualitative difference between the right shift $\si\to \si\si_{i,i+1}$ and
the left shift $\si\to \si_{i,i+1}\si$ is that the former acts on the set of positions
while the latter acts on the set of entries of word $w$.

Fix $i\in\Z$ and examine three possible cases: $(i,i+1)\subset\Z_+$,
$(i,i+1)\subset\Z_-$, and $(i,i+1)=(0,1)$.

In the first case, swapping letters $i$ and $i+1$ in $w$ reduces to swapping
the same letters in $w^+$, the components $w^-$ and $\epsi$ remaining intact.
Then we may use the fact that measure $\Q^+$ is left $q$-exchangeable.

The same argument is applicable in the second case.

In the third case, the transformation is more delicate, as it affects all
three components $w^+$, $w^-$, and $\epsi$. To describe it in detail we need to
introduce some notation.

Denote by $\varphi:W\to W$ the transformation in question (swapping
$0\leftrightarrow1$). We write $W=W^{10}\sqcup W^{01}$, where the subset
$W^{10}\subset W$ consists of the words in which $1$ is on the left of 0, and
its complement $W^{01}\subset W$ comprises the words in which 0 is on the left
of 1. Next, we consider finer splittings
$$
W^{10}=\bigsqcup_{p,k,r,l\ge0} W^{10}(p,k;r,l), \qquad
W^{01}=\bigsqcup_{p,k,r,l\ge0} W^{01}(p,k;r,l),
$$
according to the following constraint on the $\epsi$-component of $w$.

$\bullet$ Subset $W^{10}(p,k;r,l)$: on the left of the position of letter 1,
the number of $+1$'s in $\epsi$ equals $p$; between the positions of letters 1
and 0, the number of $+1$'s and $-1$'s is $k$ and $l$, respectively; on the
right of 0, the number of $-1$'s equals $r$.

$\bullet$ Subset $W^{01}(p,k;r,l)$: the same conditions, with letters 0 and 1
interchanged.

Obviously, $\varphi$ maps $W^{10}(p,k;r,l)$ onto $W^{01}(p,k;r,l)$ and vice
versa.

We need an extra notation: For $a\ge0$, $W^+_a\subset W^+$ consists of the
words $w^+$ in which letter 1 occupies position $a+1$ counting from the left.
Likewise, for $b\ge0$, $W^-_b\subset W^-$ comprises the words $w^-$ in which
letter 0 occupies position $b+1$ counting from the right.  Then
$$
W^{10}(p,k;r,l)=W^+_p\times W^-_r\times E^{10}(p,k;r,l)
$$
where $E^{10}(p,k;r,l)$ is some subset in $E$. Likewise,
$$
W^{01}(p,k;r,l)=W^+_{p+k}\times W^-_{r+l}\times E^{01}(p,k;r,l)
$$
where $E^{01}(p,k;r,l)$ is some subset in $E$.

Now, the key fact is that, in this notation, the bijection
$\varphi:W^{10}(p,k;r,l)\to W^{01}(p,k;r,l)$ factors through a triple of
transformations
$$
\varphi^+: W^+_p\to W^+_{p+k}, \quad \varphi^-:W^-_r\to W^-_{r+l}, \quad
\widetilde\varphi: E^{10}(p,k;r,l)\to E^{01}(p,k;r,l),
$$
where $\varphi^+$ moves letter 1 in $w^+$ to $k$ positions on the right,
$\varphi^-$ moves letter 0 in $w^-$ to $l$ positions on the left, and
$\widetilde\varphi$ is a transformation that does not depend on the $w^\pm$
components.

By the virtue  of $q$-exchangeability of measures $\Q^\pm$, $\varphi^+$
produces the factor $q^k$ on every $w^+\in W^+_p$, and $\varphi^-$ produces
factor $q^l$ on every $w^-\in W^-_l$. On the other hand, $\widetilde\varphi$
diminishes the total number of inversions in every binary word $\epsi\in
E^{10}(p,k;r,l)$ by $k+l+1$ and so produces the factor $q^{-k-l-1}$. Therefore,
the resulting effect of swapping $10\to01$ is the constant factor $q^{-1}$, as
it should be.

Likewise, the bijection $\varphi:W^{01}(p,k;r,l)\to W^{10}(p,k;r,l)$ produces
the desired factor $q$. This concludes the proof of the lemma.
\end{proof}

To address the issue of uniqueness with shall change a viewpoint and interpret
permutation as order. This is almost tautological for finite permutations, but
requires some care for permutations of infinite sets, because by far not every
order corresponds to a permutation.

Formally, by an {\it order\/} on set $X$ we understand a strict total order,
which is a binary relation $x\prec y$ on $X$ satisfying three conditions:
$x\prec y$ implies $x\ne y$ (the order is strict), $x\prec y$ and $y\prec z$
implies $x\prec z$ (transitivity), if $x\ne y$ then either $x\prec y$ or
$y\prec x$ (completeness).

Let $X$ be a finite or countable set, and let $\Ord(X)$ be the set of all
orders on $X$. The group $\Sym_X$ of all permutations of $X$ acts  on $X\times
X$, hence acts on $\Ord(X)$, provided we identify the order with  its graph.
The identification also enables us to topologize $\Ord(X)$ by viewing it as a
subset of $\{0,1\}^{X\times X}$. The group $\Sym_X$ acts on $\Ord(X)$ by
homeomorphisms.

Orders can be restricted from larger sets to smaller, hence for $Y\subset X$
there is a natural projection map $\Ord(X)\to\Ord(Y)$.  For $X$ countable the
space $\Ord(X)$ can be identified with the projective limit
$\varprojlim\Ord(Y)$, where  $Y$ ranges over finite subsets of $X$.

Let $I$ be an interval in $\Z$, possibly coinciding with $\Z$ itself. Every
permutation $\si\in\Sym_I$ determines an order $\prec$ on $I$ by setting
\begin{equation}\label{eq20}
i\prec j \quad  \textrm{iff} \quad \si^{-1}(i)<\si^{-1}(j).
\end{equation}
Equivalently, writing $\si$ as permutation word with positions labelled by $I$,
$i\prec j$ means that letter $i$ occurs in $w$ before letter $j$. In this way,
we get a natural map $\Sym_I\to\Ord(I)$ that intertwines the left action of
$\Sym_I$ on itself with its canonical action on $\Ord(I)$. For $I\ne\Z$ this
map is an embedding while for $I=\Z$ it is not. The reason is that the order
set $(\Z,<)$ has a nontrivial group of symmetries formed by shifts, which
implies that two permutations $\si,\tau\in\Sym_\Z=\Sym$ induce the same order
if and only if $\si=s^{(b)}\tau$ with some $b\in\Z$. In the case $I\ne\Z$  this
effect disappears, for then the ordered set $(I,<)$ has only trivial
automorphisms.

If $I$ is finite, the map $\Sym_I\to \Ord(I)$ is a bijection. But for $I=\Z$ or
a semi-infinite interval $I=Z_{\geq a},\Z_{\leq a}$ the orders coming from
permutations constitute a very small subset in $\Ord(I)$.

For arbitrary interval $I\subseteq\Z$ and any $q>0$, the notion of a
$q$-exchangeable measure on $\Ord(I)$ is introduced in complete analogy with
Definition \ref{def2}, only now we do not need to distinguish between right and
left actions of transpositions, since the action is unique.

\begin{lemma}\label{lemma3}
For every interval $I\subseteq\Z$ and  $q>0$, there exists a unique
$q$-exchangeable probability measure $\mu_{I,q}$ on\/ $\Ord(I)$.
\end{lemma}

\begin{proof}
For finite $I$ this is a property of the Mallows measure \eqref{MalFin}, since
$(I,<)$ is isomorphic to $(\{1,\dots,n\},<)$, for $n=\#I$.

Assume now $I=\Z$. The space $\Ord(\Z)$ coincides with the projective limit
space $\varprojlim\Ord(I)$, where $I$ ranges over the set of finite intervals.
For two finite intervals $I\subset J$, the restriction map $\Ord(J)\to\Ord(I)$
is consistent with $q$-exchangeability and so maps $\mu_{J,q}$ to $\mu_{I,q}$.
Appealing to Kolmogorov's extension theorem shows that $\mu_{\Z,q}$ exists and
is unique.

For semi-infinite interval $I\subset \Z$ the argument is exactly the same.
\end{proof}

\begin{lemma}\label{lemma4}
There exists at most one left $q$-exchangeable probability measure on
$\Sym^\bal$.
\end{lemma}

\begin{proof}
Restricting the map $\Sym\to\Ord(\Z)$ to $\Sym^\bal$ we get an embedding
$\Sym^\bal\to\Ord(\Z)$. The latter map is continuous and hence Borel, and the
image of $\Sym^\bal$ is a Borel subset in $\Ord(\Z)$ (the latter claim follows,
e.g., from the fact that both $\Sym^\bal$ and $\Ord(\Z)$ are standard Borel
spaces). Therefore, two distinct left $q$-exchangeable probability measures on
$\Sym^\bal$ would give rise to two distinct $q$-exchangeable probability
measures on $\Ord(\Z)$, which is impossible by the virtue of Lemma \ref{lemma3}.
\end{proof}

The proof of Theorem \ref{thm2} is thus completed.
\end{proof}

\begin{remark}\label{rem1}
We have seen that the Mallows measure $\Q$ on $\Sym^\bal$ is obtainable from
the $q$-exchangeable measure $\mu_{\Z,q}$ on the space $\Ord(\Z)$. However, the
latter measure exists for every $q>1$ while $\Q$ is defined only for $0<q<1$.
The obvious explanation is that for $q\ge1$ the measure $\mu_{\Z,q}$ is no
longer supported by $\Sym^\bal\subset\Ord(Z)$.

For $q>1$, $\mu_{\Z,q}$ is still supported by
permutations. Namely, by non-balanced permutations of the type
$\tau\si$, where
$\si$ ranges over $\Sym^\bal$ and $\tau\in\Sym$ is the reflection map
$\tau(i)=-i$.

For $q=1$, the measure $\mu_{\Z,q}=\mu_{\Z,1}$ is the only exchangeable
probability measure on $\Ord(\Z)$. In this case, the  order structure on
$\Z$ plays no role, we simply regard $\Z$ as a countable set. Following a
well-known recipe, the random exchangeable order can be defined by declaring
$i\prec j$ iff $\xi_i<\xi_j$, where $\xi_i$'s are independent, uniform $[0,1]$
random variables. It follows that the type of the exchangeable order is almost
surely $({\mathbb Q}, <)$. Remarkably, for $q\ne1$ the situation is different,
in that the $q$-exchangeable order is of the type $(\Z,<)$.
\end{remark}

\section{A construction from independent geometric variables}
\label{section4}

Fix $q\in(0,1)$ and let $\Si$ denote the random $q$-exchangeable balanced
permutation of $\Z$, i.e. the random element of $\Sym^\bal$ distributed
according to Mallows' measure $\Q$ with parameter $q$. Likewise, let $\Si_+$
and $\Si_n$ denote the random $q$-exchangeable permutations of $\Z_+$ and
$\{1,\dots,n\}$, respectively. There is a simple algorithm to construct $\Si_n$
from independent truncated geometric variables.  In \cite{Qexch}, we described
a ``one-sided infinite'' extension of this algorithm, called {\it
$q$-shuffle\/}, to generate $\Si_+$. As an application of the $q$-shuffle, we
obtained an isomorphism of the measure space $(\Sym_{\Z_+},\Q^+)$ with the
space $\{0,1,2,\dots\}^{\Z_+}$ equipped with a product of geometric
distributions. The aim of this section is to derive a ``two-sided infinite''
analog of these results: we shall introduce an algorithm of generating $\Si$
from  infinitely many copies of a geometric  random variable. In comparison
with the interlacing construction, one advantage of the new approach is that it
makes transparent some stationarity property of $\Q$.

For $\si\in\Sym$, a pair of positions $(i,j)\in\Z\times\Z$ is an {\it
inversion\/} in $\sigma$ if $i<j$ and $\sigma(i)>\sigma(j)$. If $(i,j)$ is an
inversion, we say that it is a {\it left inversion\/} for $j$, and a {\it right
inversion\/} for $i$. Introduce the  counts of left and right inversions,
$$
\ell_i: =\#\{j: j<i,~\sigma(j)>\sigma(i) \}, \quad r_i:=\#\{j:
j>i,~\sigma(j)<\sigma( i)\}, \qquad i\in\Z,
$$
respectively (of course, for the general $\si\in\Sym$ these quantities may be
infinite). The following easy proposition relates these notions to the
admissibility condition.

\begin{proposition}\label{prop2}
If $\si$ is admissible, then $\ell_{i},\,r_{i}<\infty$  for all $i\in \Z$.
Conversely, if this condition holds for some index $i\in\Z$, then it holds for
all $i$ and $\si$ is admissible.
\end{proposition}

\begin{proof}
Consider the permutation matrix $A=A(\si)$. We assume that its rows  are
enumerated from top to bottom, and the columns from left to right. Then
$\ell_i$ equals the number of 1's in the lower left corner of $A$ formed by
positions $(a,b)$ satisfying inequalities $a>\si(i)$, $b<i$. Likewise, $r_i$ is
the number of 1's in the upper right corner determined by the opposite
inequalities $a<\si(i)$, $b>i$. For any two lower left corners, the
corresponding numbers of 1's always differ by a finite quantity, and the same
holds for upper right corners. This makes the claim of the proposition evident.
\end{proof}

Obviously,
\begin{equation}\label{inv1}
\inv(\sigma)=\sum_{i\in \Z}\ell_i=\sum_{i\in \Z}r_i
\end{equation}
is the total number of inversions in $\sigma$; for balanced permutations this
quantity is finite if and only if $\si\in\Sym^\fin$.

The left and right inversion counts are defined similarly for finite
permutations. A well-known fact is that a permutation $\si\in\Sym_n$ can be
uniquely encoded in the sequence of its right inversion counts
$$
(r_1,\dots,r_n)\in \{0,\dots,n-1\}\times\dots\times\{0,1\}\times\{0\}.
$$
For instance, the permutation word $(3,1,2,4)$ (i.e. permutation in the one-row
notation) corresponds to the sequence $(2,0,0,0)$.
Similarly, $\si\in\Sym_n$ can be uniquely encoded in the sequence of its left
inversion counts
$$
(\ell_1,\dots,\ell_n)\in\{0\}\times\{0,1\}\times\dots\times\{0,\dots,n-1\}.
$$
The correspondence inverse to the latter amounts to the following well-known algorithm:

\medskip
\noindent {\bf Elimination algorithm}: 1) Given a sequence $(r_1,\dots,r_n)$,
construct a permutation word $w=w_1\dots w_n$ recursively, as follows. At step
1 set $w_1=r_1+1$, and eliminate integer $r_1+1$ from the list $1,\dots,n$. At
step 2 let $w_2$ be equal to the $(r_2+1)$th smallest entry of the reduced
list, and eliminate the entry from the list, etc. For instance, from
$(2,0,0,0)$ we derive $w_1=2+1=3$, and remove $3$ from the initial list
$1,2,3,4$ to obtain $1,2,4$. At step 2 we find $w_2=1$ (which is the $(0+1)$st
smallest element $1,2,4$) and further reduce the list to $2,4$, etc.

2) Likewise, one can construct  $w$ from $(\ell_1,\dots,\ell_n)$. In this case,
we first determine $w_n$ from $\ell_n$, then find $w_{n-1}$, etc.

\medskip

For $1\leq i\leq n$ let $L_{n,i}$ and $R_{n,i}$  be the random counts of left
and right  inversions in the Mallows permutation  $\Sigma_n$. It is immediate
from the  above bijections and \eqref{MalFin} that
\begin{itemize}
\item[(i)]  $R_{n,i}$  are independent for $i=1,\dots,n$,
and have the truncated geometric distributions
$$
\Prob(R_{n,i}=k)=\frac{q^{k}}{[n-i]_q}, \qquad k=0,\dots,n-i,
$$
\item[(ii)]
$L_{n,i}$  are independent for $i=1,\dots,n$,
and have the truncated geometric distributions
$$
\Prob(L_{n,i}=k)=\frac{q^{k}}{[i]_q}, \qquad k=0,\dots,i-1,
$$
where $[m]_q:= \sum_{k=0}^{m-1} q^k$ denotes the $q$-number.
\end{itemize}
\medskip

Sending $n\to\infty$ in (i) leads to the $q$-shuffle algorithm \cite{Qexch}
mentioned above:

\medskip

\noindent{\bf The $q$-shuffle algorithm}. Let $R_1,R_2,\dots$ be independent,
identically distributed random variables with the geometric distribution on
$\{0,1,\dots\}$ with parameter $q$,
$$
\Prob(R_i=k)=(1-q)q^{k}, \qquad k=0,1,2,\dots, \quad i=1,2,\dots\,.
$$
Set $w_1=R_1+1$. Inductively, choose for $w_i$ the $(R_i+1)$th smallest element
of $\Z_+\setminus\{w_1,\dots,w_{i-1}\}$. Eventually every element of $\Z_+$
will be chosen at some step, because $R_i=0$ for infinitely many $i$ almost
surely. Therefore the  procedure yields almost surely a word $w=w_1w_2\dots$
corresponding to a permutation $\si\in\Sym_{\Z_+}$. Note that its right
inversion counts are just $R_1,R_2,\dots$.

\medskip

\begin{remark}
One can prove that the left inversion counts of $\Sigma_+$ are independent
random variables with truncated geometric distributions, as in (ii). However,
generation of $\Si_+$ through the left inversion counts is a more difficult
task, for a direct extension of the second version of the elimination algorithm
no longer works: the right-most letter does not exist. A similar difficulty
arises with two-sided infinite words, since for them there is no qualitative
difference between right and left inversion counts. To resolve this difficulty
we will exploit below a more sophisticated  procedure.
\end{remark}

Consider the product space $X:=\{0,1,2,\dots\}^\Z$ and equip it with the
product measure $\G:=\bigotimes_{i\in\Z}G$, where $G$ stands for the geometric
distribution on $\{0,1,2,\dots\}$ with a fixed parameter $q\in(0,1)$, that is,
\begin{equation}\label{eq23}
G(n)=(1-q)q^n, \qquad n=0,1,2,\dots\,.
\end{equation}
Let $\psi:\Sym^\bal\to X$ be the map defined by the right inversion counts:
$\psi(\si)=(r_i: i\in\Z)$.

\begin{theorem}\label{thm3}
The map $\psi$ provides an isomorphism of the measure space $(\Sym^\bal,\Q)$
onto the product measure space $(X,\G)$.
\end{theorem}

\noindent{\bf Comment.} Before proceeding to the proof, note that under the
correspondence $\psi:\si\to (r_i)$, the alternative ``either $w_i<w_{i+1}$ or
$w_i>w_{i+1}$'' for the word $w\leftrightarrow\si$ translates as ``either
$r_i\le r_{i+1}$ or $r_i>r_{i+1}$", and the transformation $T_{i,i+1}:\si\to
\si\si_{i,i+1}$ turns into the transformation $T'_{i,i+1}: X\to X$ that does
not affect coordinates $r_j$ with $j\ne i,\,i+1$ and reduces to
$$
(r_i,r_{i+1}) \to \begin{cases} (r_{i+1}+1,\, r_i), & r_i\le r_{i+1}\\
(r_{i+1},\,r_i-1), & r_i>r_{i+1}.
\end{cases}
$$
Thus, $\G$ is transformed under the action of $T'_{i,i+1}$ in the same way as
$\Q$ is transformed under $T_{i,i+1}$, i.e. the Radon-Nikodym derivative equals
$q^{\pm1}$, depending on whether $r_i\le r_{i+1}$ or $r_i>r_{i+1}$. This
observation agrees with the claim of the theorem but is not yet enough for the
proof.
\medskip

\begin{proof}
Regard $(r_i)$ as a two-sided infinite sequence of random variables defined on
the probability space $(\Sym^\bal,\Q)$.

\begin{lemma}\label{lemma5}
The random sequence $(r_i)$ is stationary with respect to shifts $i\to i+n$.
\end{lemma}

\begin{proof}
Use for a moment a more detailed notation $r_i(\si)$. We have
$$
r_{i+n}(\si)=r_i(s^{(n)}\si s^{(-n)}), \qquad n\in\Z.
$$
On the other hand, the map $\si\to s^{(n)}\si s^{(-n)}$ leaves invariant the
subgroup $\Sym^\bal\subset\Sym^\adm$ and preserves the $q$-exchangeability
property. Due to uniqueness of measure $\Q$, it is invariant under this map,
whence the assertion of  the lemma.
\end{proof}

\begin{lemma}\label{lemma6}
For every finite sequence of integers $i_1<\dots<i_k$, the joint law of
$(r_{i_1},\dots,r_{i_k})$ is the product measure $G\otimes\dots\otimes G$.
\end{lemma}

\begin{proof}
By virtue of Lemma \ref{lemma5}, it suffices to prove that, as $n\to+\infty$,
the limit law for $(r_{i_1+n},\dots,r_{i_k+n})$ exists and coincides with
$G\otimes\dots\otimes G$.

Use the encoding $\si\leftrightarrow w\leftrightarrow(w^+,w^-,\epsi)$. Given
$\epsi$, the conditional distribution of $(r_{i_1+n},\dots,r_{i_k+n})$
stabilizes as $n\to+\infty$ and coincides with $G\otimes\dots\otimes G$.
Indeed, this immediately follows from the definition of the encoding and the
$q$-shuffle algorithm generating $w^+$. This implies the desired claim.
\end{proof}

By virtue of Lemma \ref{lemma6}, it suffices to prove that the map $\psi$ is
injective. We do this in the next two lemmas: Lemma \ref{lemma7} says how to
reconstruct a balanced permutation from the two sequences $(r_i)$ and
$(\ell_i)$, and Lemma \ref{lemma8} describes an algorithm which expresses
$(\ell_i)$ through $(r_i)$.

\begin{lemma}\label{lemma7}
For permutation $\sigma\in\Sym^\bal$ we have
\begin{equation}\label{eq7}
\sigma(i)=i+r_i-\ell_i, \qquad i\in\Z.
\end{equation}
\end{lemma}

\begin{proof}
It is convenient to prove a slightly more general claim: for each
$\si\in\Sym^\adm$
\begin{equation}\label{eq6}
r_i-\ell_i+i-\si(i)=b(\si), \qquad i\in\Z
\end{equation}
(recall that $b(\si)$ denotes the balance of $\si$). Since $\bal(\si)=0$ for
$\si\in\Sym^\bal$, \eqref{eq6} will imply \eqref{eq7}.
Indeed, by \eqref{eq5}, for each $n\in\Z$ we have
$$
b(s^{(n)}\si)=b(\si)+n, \qquad \si\in\Sym^\adm.
$$
On the other hand, if we replace $\si$ by $s^{(n)}\si$, then all counts $r_i$,
$\ell_i$ will not change, while $i-\si(i)$ will transform to $i-\si(i)+n$. Thus,
\eqref{eq6} is consistent with the replacement $\si\to s^{(n)}\si$. Now  fix an
arbitrary $i$, take $n=\si(i)-i$, and replace $\si$ by $\si':=s^{(n)}\si$. Then
$\si'(i)=i$, which in turn implies that $b(\si')=r_i-\ell_i$, by the very
definition of balance and the left and right inversion counts. Consequently,
\eqref{eq6} holds true for $\si'$.
\end{proof}

Fix $\si\in\Sym^\bal$ and let, as usual, $w$ stand for the corresponding
two-sided infinite permutation word. Observe that $r_j+1$ is the rank of $w_j$
in the left-truncated word $w_jw_{j+1}\dots$, with respect to the canonical
order on $\Z$. (That is, the rank equals $k$ if $w_j$ is the $k$th minimal element
among $w_j$, $w_{j+1}$, \dots\,.) More generally, for $i\le j$, define
$r^{(i)}_j$ to be one smaller the rank of $w_j$ among $w_i\dots w_j\dots$. Evidently,
$r^{(j)}_j=r_j$.

\begin{lemma}\label{lemma8}
{\rm(i)} Fix $j\in\Z$. For  $i\le j$, the quantity $r^{(i)}_j$ can be
determined from the finite sequence $r_j, r_{j-1},\dots,r_i$ by means of the
recursion
\begin{equation}\label{eq21}
r^{(j)}_j=r_j, \qquad r^{(i-1)}_j=r^{(i)}_j+\mathbf1\left(r^{(i)}_j\ge
r_{i-1}\right).
\end{equation}
{\rm(ii)} The left inversion counts can be determined from the quantities
$\left(r^{(i)}_j: i\le j\right)$  by
\begin{equation}\label{eq22}
\ell_j=\sum_{i:\, i<j}\mathbf1\left(r^{(i)}_j<r_i\right)=\sum_{i:\, i<j}
\mathbf1\left(r^{(i)}_j=r^{(i+1)}_j\right).
\end{equation}
\end{lemma}

\begin{proof}
(i) We have
$$
r^{(i-1)}_j=r^{(i)}_j+\mathbf1(w_{i-1}<w_j).
$$
But $w_{i-1}<w_j$ is equivalent to $r^{(i-1)}_{i-1}\le r^{(i)}_j$, and
$r^{(i-1)}_{i-1}=r_{i-1}$.

(ii) We have
$$
\ell_j=\sum_{i:\, i<j}\mathbf1(w_i>w_j),
$$
but $w_i>w_j$ is equivalent to $r^{(i)}_i>r^{(i)}_j$ (i.e., $r_i>r^{(i)}_j$)
and also to $r^{(i)}_j=r^{(i+1)}_j$.
\end{proof}

This concludes the proof of Theorem \ref{thm3}.

\end{proof}

\begin{corollary}
The Mallows measures on $\Sym^\bal$ corresponding to any two distinct values of
parameter $q$ are disjoint\/ {\rm(}mutually singular\/{\rm)}.
\end{corollary}

\begin{proof}
Immediate by virtue of classical Kakutani's theorem \cite{Kakutani} about product
measures.
\end{proof}

The latter can also be seen from the law of large numbers:
$\ell_1+\dots+\ell_n$ under $\Q$ is asymptotic to $nq/(1+q)$.

\section{The distribution of displacements}\label{section5}

Let as above $\Q$ be the Mallows measure on $\Sym^\bal$ with parameter
$q\in(0,1)$ and let $\Si$ be the random permutation of $\Z$ with law $\Q$. Consider
the two-sided infinite random sequence of {\it displacements\/}
\begin{equation}\label{eq12}
D_i:=\Si(i)-i, \qquad i\in\Z.
\end{equation}
The sequence $(D_i)$ is stationary in ``time'' $i\in\Z$, because a shift of
parameter $i$ amounts to a measure preserving transformation of the basic
probability space $(\Sym^\bal,\Q)$ --- conjugation of a balanced permutation by
a shift $s^{(n)}$, cf. Lemma \ref{lemma5}.

Below we use a nonstandard notation for some particular $q$-Pochhammer symbols:
$$
\langle n\rangle_q:=(q;q)_n=\prod_{k=1}^n (1-q^k), \qquad \langle
\infty\rangle_q:=(q;q)_\infty=\prod_{k=1}^\infty(1-q^k).
$$

In the present section, we will compute the one-dimensional marginal
distribution of $(D_i)$. In the next section, we will describe the
finite-dimensional distributions.

\begin{theorem}\label{thm4}
For any fixed $j\in\Z$, the distribution of displacement $D:=D_j$ is given by
\begin{equation}\label{Form1}
\Prob(D=d)=   (1-q)\langle \infty\rangle_q \sum_{\{r,\,\ell\,\ge 0:\,
r-\ell=d\}} \frac{q^{r\ell+r+\ell}}{\langle r\rangle_q \langle\ell\rangle_q},
\qquad d\in\Z.
\end{equation}
\end{theorem}

This distribution is symmetric about 0. It can be expressed through the basic
geometric series ${}_0\phi_1$ (see \cite{Gasper}),
$$
\Prob(D=d)=\frac{(1-q)(q;q)_\infty
q^d}{(q;q)_d}{}_0\phi_1(-;q^{d+1};q,q^{d+3}), \qquad d=0,1,2,\dots,
$$
or through a $q$-Bessel function, see  \cite[Ex. 1.24]{Gasper}.

We will give two proofs of the theorem.

\begin{proof}[First proof]
Let $(R_i)$ and $(L_i)$ be the sequences of right and left inversion counts for
$\Si$. By Lemma \ref{lemma7}, $D_j=R_j-L_j$. To compute the distribution of
$R_j-L_j$ we apply Lemma \ref{lemma8}. Set
$$
x_0=r^{(j)}_j, \quad x_1=r^{(j-1)}_j, \quad x_2=r^{(j-2)}_j, \dots,
$$
where we use the notation of Lemma \ref{lemma8}. Since $R_j, R_{j-1},R_{j-2},
\dots$ are  independent random variables with geometric distribution $G$, claim
(i) of Lemma \ref{lemma8} implies that, given $R_j=r$, the sequence
$x_0,x_1,\dots$ forms a Markov chain on $\{r,\,r+1,\,r+2,\,\dots\}$ with
initial state $x_0=r$, 0\,-1 increments and  transition probabilities
\begin{equation}\label{transprob}
\Prob(k\to k)=q^{k+1}, \quad \Prob(k\to k+1)=1-q^{k+1}, \qquad
k=r,\,r+1,\,r+2,\,\dots\,.
\end{equation}
Next, by claim (ii) of the lemma, $L_j$ equals the total number  of
0-increments (this number is finite almost surely by Lemma \ref{prop3}). It
follows that the (conditional) probability generating function for $L_j$ has
the form
\begin{equation}\label{genefu}
\sum_{\ell=0}^\infty\Prob(L_j=\ell\mid R_j=r)x^l=\prod_{k=r}^\infty
\frac{1-q^{k+1}}{1-xq^{k+1}}.
\end{equation}

The coefficient by $x^\ell$ is extracted from \eqref{genefu} using the Eulerian
identity (see e.g. \cite[(1.3.15)]{Gasper})
$$
\prod_{m=0}^\infty (1-yq^m)^{-1}=\sum_{n=0}^\infty \frac{y^n}{\langle
n\rangle_q},
$$
where we substitute $y=xq^{r+1}$. This gives
\begin{equation}\label{RgivL}
\Prob(L_j=\ell\mid R_j=r)= q^{\ell(r+1)}\frac{\langle\infty\rangle_q}{ \langle
r\rangle_q\langle\ell\rangle_q}.
\end{equation}
Since the distribution of $R_j$ is geometric we find the joint distribution of
$(R_j,L_j)$:
\begin{equation}\label{eq8}
\Prob(R_j=r,L_j=\ell)
=(1-q)q^{r\ell+r+\ell}\frac{\langle\infty\rangle_q}{\langle
r\rangle_q\langle\ell\rangle_q}, \qquad r,\ell\in\Z.
\end{equation}
Finally, the distribution of $D=D_j=R_j-L_j$ follows by summation.
\end{proof}

\begin{proof}[Second proof]
By stationarity, it suffices to find the distribution of $\Si(1)-1$. Next, it
will be convenient for us to replace $\Si$ with the inverse permutation
$\Si^{-1}$, which has the same law, see Corollary \ref{cor1}. Thus, we will
deal with $\Si^{-1}(1)-1$.

Let $w=(w_i)_{i\in\Z}$ be the random word corresponding to $\Si$. We encode $w$
by the triple $(w^+,w^-,\la)$, see the definitions in \eqref{eq9},
\eqref{eq13}, \eqref{eq10}, and \eqref{eq11} above. Denote by $B+1=1,2,\dots$
the position of 1 in $w^+$. Then, in the notation of \eqref{eq9}, the position
of 1 in $w$ is $i_{B+1}$, which is the same as $B+1-\la_{B+1}$, as is seen from
\eqref{eq13}. On the other hand, this position is just
$\Si^{-1}(1)$. Therefore, we may write
$$
\Si^{-1}(1)=B+1-C, \qquad C:=\la_{B+1}.
$$
To find the distribution of
$$
\Si^{-1}(1)-1=B-C
$$
we will compute the probabilities
$$
\Prob(B=b, \, C=c), \qquad b,c=0,1,2,\dots\,.
$$

Recall that $w^+$ and $\la$ are independent. Therefore,
$$
\Prob(B=b, \, C=c)=\Prob(\textrm{1 occupies in $w^+$ position
$b+1$})\cdot\Prob(\la_{b+1}=c).
$$
The first factor in the right-hand side is determined from the $q$-shuffle
algorithm generating $w^+$. This gives
$$
\Prob(\textrm{1 occupies in $w^+$ position $b+1$})=q^b(1-q).
$$
In the second factor $\Prob(\la_{b+1}=c)$, the probability is understood according
to the distribution \eqref{eq3}. Therefore,
$$
\Prob(\la_{b+1}=c)=\langle\infty\rangle_q\sum_\la q^{|\la|}
$$
summed over diagrams $\la$ of shape specified on Figure 1. The sum in question
is computed using the generating function for Young diagrams in a strip:
$$
\sum_{\nu=(\nu_1\ge\dots\ge\nu_a\ge0)}q^{|\nu|}=\frac1{\langle a\rangle_q};
$$
we apply this formula for $\nu=\la^+$ and $\nu=(\la^-)'$ (the transposed
diagram). This gives

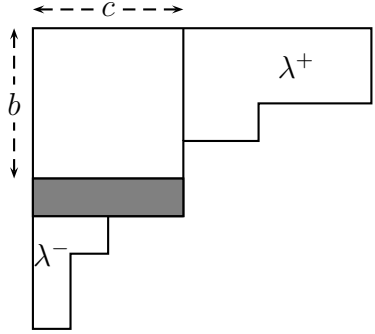
\begin{figure}

\vskip1truecm

\psset{unit=0.5cm}
\begin{center}
\begin{pspicture}(-1,-1)(10,-10)
\pspolygon(0,0)(9,0)(9,-2)(6,-2)(6,-3) (4,-3) (4,-4) (4,-5)
(2,-5)(2,-6)(1,-6)(1,-8)(0,-8) \psline(0,-4)(4,-4) \psline(4,0)(4,-3)
\psline(0,-5)(4,-5) \pspolygon[fillstyle=solid,
fillcolor=gray](0,-4)(4,-4)(4,-5)(0,-5) \psset{linestyle=dashed, arrows=<->}
\psline(0,0.5)(4,0.5) \rput(2.0, 0.5) {\psframebox*[boxsep=false] {$c$}}
\rput(7,-1){$\lambda^+$} \rput(0.5, -6){$\lambda^-$} \psline(-0.5,0)(-0.5,-4)
\rput*(-0.5,-2){$b$}
\end{pspicture}
\end{center}
\caption{One-row decomposition of the diagram.}
\end{figure}

$$
\Prob(\la_{b+1}=c)=\langle\infty\rangle_q\,\frac{q^{(b+1)c}} {\langle
b\rangle_q \langle c\rangle_q}.
$$

Thus,
$$
\Prob(B=b,\,C=c)=(1-q)\langle\infty\rangle_q\,\frac{q^{bc+b+c}} {\langle
b\rangle_q \langle c\rangle_q}
$$
and finally
$$
\Prob(\Si^{-1}(1)-1=d)=\Prob(B-C=d)=(1-q)\langle\infty\rangle_q\,\sum_{b,\,c\ge0:\,
b-c=d}\frac{q^{bc+b+c}} {\langle b\rangle_q \langle c\rangle_q},
$$
which is the same as \eqref{Form1}.
\end{proof}

\begin{remark}
The distribution \eqref{Form1} has  exponentially decaying tails
$$
\Prob(|D|>m)\asymp q^m,     \qquad m\to\infty.
$$
Indeed, the lower bound follows from \eqref{Form1} while the upper bound is
easy from
$$
\Prob(|D_j|>m)\le \,\max(\Prob(R_j>m), \,\Prob(L_j>m))\le 2\Prob(R_j>m)=2q^m.
$$
This allows us to estimate the size of fluctuation of $\Sigma$ about the
identity permutation. Using
$$\
\Prob(|D_n|>(1+\epsilon)\log_{1/q}|n|)<|n|^{-1-\epsilon}
$$
and applying  the Borel-Cantelli lemma we obtain
$$
\limsup_{|n|\to\infty}\frac{|D_n|}{\log_{1/q}|n|}\leq 1 \qquad {\rm a.s.}
$$
We conjecture that this bound is sharp.
\end{remark}

\section{Finite-dimensional distributions}\label{section6}

We proceed with deriving a multivariate distribution for the displacements
\eqref{eq12}.

\begin{theorem}
For $k=1,2,\dots$ and integers $d_1\le\dots\le d_k$
\begin{multline*}
\Prob(D_1=d_1,\dots, D_k=d_k)\\
=(1-q)^kq^{-k(k+1)/2} \langle\infty\rangle_q \prod_{m=2}^{k} \langle
d_m-d_{m-1}\rangle_q \sum \frac{q^{\sum_{1\leq i\leq j\leq
k}(b_i+1)(a_j+1)}}{\langle b_1\rangle_q\dots \langle b_k \rangle_q \langle
a_1\rangle_q \dots \langle a_k \rangle_q},
\end{multline*}
where the summation is over all nonnegative integers $a_1,b_1,\dots,a_k,b_k$
which satisfy the constraints
\begin{equation}\label{constraints}
(b_1+\dots+b_m)-(a_m+\dots+a_k)=d_m,  \qquad m=1,\dots,k.
\end{equation}
\end{theorem}

\noindent{\bf Comments.} (i) In the case $k=1$ the product over $m$ is empty
and the result agrees with Theorem \ref{thm4}.

(ii) The constraints $d_1\le \dots\le d_k$ are not substantial and only
introduced to simplify the formula. These inequalities are equivalent to
$\Si(1)<\dots<\Si(k)$, so that the general case can be handled by introducing
the additional factor
$$
q^{{\rm inv}(d_1+1,\dots,d_k+k)}=q^{\inv(\Si_1,\dots,\Si_k)}
$$
implied by the $q$-exchangeability; here ``$\inv$'' stands for the number of
inversions.

(iii) By stationarity, the distribution does not change if we simultaneously
shift all indices $1,\dots,k$ by a constant. However, we did not manage to
write a reasonable formula for $(D_{i_1},\dots, D_{i_k})$ with arbitrary
indices $i_1<\dots<i_k$.

(iv) Excluding variables $b_1,\dots,b_k$ from relations \eqref{constraints},
the resulting inequalities on the remaining $k$ variables $a_1,\dots,a_k$ take
the form
\begin{gather*}
0\le a_1\le d_2-d_1 \\
0\le a_2\le d_3-d_2\\
\vdots\\
0\le a_{k-1} \le d_k-d_1\\
a_k\ge\max(0,\, -a_1-\dots-a_{k-1}-d_1),
\end{gather*}
which shows that the summation runs over a domain with only one of free
variables assuming infinitely many values.

\medskip

\begin{proof}
We generalize the second proof of Theorem \ref{thm4}. Let us compute the
probability
\begin{equation}\label{eq15}
\Prob(\Si^{-1}(1)-1=d_1,\,\dots,\, \Si^{-1}(k)-k=d_k).
\end{equation}

Consider again the encoding $(w^+,w^-,\la)$ of the random word $w$ associated
with $\Si$. Let $x_1,\dots,x_k$ be the positions of letters $1,\dots,k$ in
$w^+$. Then the positions of the same letters in $w$ are
$$
\Si^{-1}(1)=x_1-y_1, \,\dots, \,\Si^{-1}(k)=x_k-y_k,
$$
where we set
$$
y_1=\la_{x_1},\, \dots,\, y_k=\la_{x_k}.
$$

The assumption $d_1\le\dots\le d_k$ means $1\le x_1<\dots<x_k$, which entails
$$
y_1\ge\dots\ge y_k\ge0.
$$
Now pass to new variables $b_1,\dots,b_k,a_1,\dots,a_k$ by setting
\begin{equation}\label{eq14}
\begin{aligned}
x_1 &=b_1+1  &\quad y_1& =a_1+\dots +a_k\\
x_2 &=b_1+b_2+2  & \qquad y_2& =a_2+\dots+a_k\\
 & \phantom{a}\vdots  & \qquad &\phantom{a}\vdots\\
 x_k &=b_1+\dots+b_k+k  & \qquad y_k & = a_k
\end{aligned}
\end{equation}
Then the above inequalities imposed on $x_1,\dots,x_k, y_1,\dots,y_k$ just mean
that the new variables are nonnegative, and conditions
$$
\Si^{-1}(m)-i=d_m, \qquad m=1,\dots,k
$$
take the form \eqref{constraints}.

Introduce the set of Young diagrams
\begin{equation}\label{eq16}
\Lambda(x_1,\dots,x_k;y_1,\dots,y_k)=\{\la: \la_{x_1}=y_1, \dots,
\la_{x_k}=y_k\}.
\end{equation}
The probability \eqref{eq15} can be written in the form
\begin{equation}\label{eq19}
\sum_{b_1,\dots,b_k,\,a_1,\dots,a_k}
P_1(x_1,\dots,x_k)P_2(x_1,\dots,x_k;y_1,\dots,y_k),
\end{equation}
with summation
over nonnegative integers $b_1,\dots,b_k,a_1,\dots,a_k$ subject to
constraints \eqref{constraints}, where $P_1(x_1,\dots,x_k)$ is the probability
that letters $1,\dots,k$ occupy positions $x_1,\dots,x_k$ in $w^+$, and
$$
P_2(x_1,\dots,x_k;y_1,\dots,y_k)
=\langle\infty\rangle_q\sum_{\la\in\Lambda(x_1,\dots,x_k;y_1,\dots,y_k)}
q^{|\la|}
$$
is the probability that the random diagram $\la$ with law \eqref{eq3} falls
into the subset \eqref{eq16}. This probability is computed in the next lemma.

\begin{figure}\label{Fig2}
\psset{xunit=1pt} \psset{yunit=1pt} \psset{runit=1pt}
\begin{center}
\begin{pspicture}(0,0)(250,300)

\pspolygon(20,300)(180,300)(180,240)(20,240) \psline(20,250)(180,250)

\pspolygon(20,240)(140,240)(140,180)(20,180) \psline(20,190)(140,190)

\put(50,160){\circle*{2}} \put(60,160){\circle*{2}} \put(70,160){\circle*{2}}

\pspolygon(20,140)(100,140)(100,90)(20,90) \psline(20,100)(100,100)

\pspolygon[fillstyle=vlines,fillcolor=gray,hatchangle=-45](180,300)(220,300)(220,290)(210,290)
(210,280)(195,280)(195,270)(190,270)(190,260)(185,260)(185,250)(180,250)

\pspolygon[fillstyle=vlines,fillcolor=gray,hatchangle=-45](140,240)(180,240)(180,230)(170,230)(170,215)
(155,215)(155,200)(150,200)(150,190)(140,190)

\pspolygon[fillstyle=vlines,fillcolor=gray,hatchangle=-45](100,140)(120,140)(120,130)(115,130)(115,110)
(105,110)(105,100)(100,100)

\pspolygon[fillstyle=vlines,fillcolor=gray,hatchangle=-45](20,90)(100,90)
(100,60)(70,60)(70,50) (60,50)(60,20)(30,20)(30,5)(20,5)

\pspolygon[fillstyle=solid, fillcolor=gray] (20,250)(180,250)(180,240)(20,240)

\pspolygon[fillstyle=solid, fillcolor=gray] (20,190)(140,190)(140,180)(20,180)

\pspolygon[fillstyle=solid, fillcolor=gray] (20,100)(100,100)(100,90)(20,90)

\psline[linestyle=dotted, dotsep=1pt](180,190)(180,240)
\psline[linestyle=dotted, dotsep=1pt](180,190)(140,190)

\psline[linestyle=dotted,dotsep=1pt](120,100)(120,130)
\psline[linestyle=dotted,dotsep=1pt](120,100)(100,100)

\psset{linestyle=dashed,arrows=<->} \psline(140,175)(180,175) \put(158,179)
{$a_1$} \psline(100,145)(120,145) \put(108,149) {$a_{k-1}$}
\psline(20,2)(100,2) \put(58,6){$a_k$}

    \psline (15,250)(15,300)
\put(2,273){$b_1$} \psline(15,190)(15,240) \put(2,213){$b_2$}
\psline(15,190)(15,240) \psline(15,100)(15,140) \put(2,117){$b_k$}
\end{pspicture}
\end{center}
\caption{Multirow decomposition of a diagram}
\end{figure}
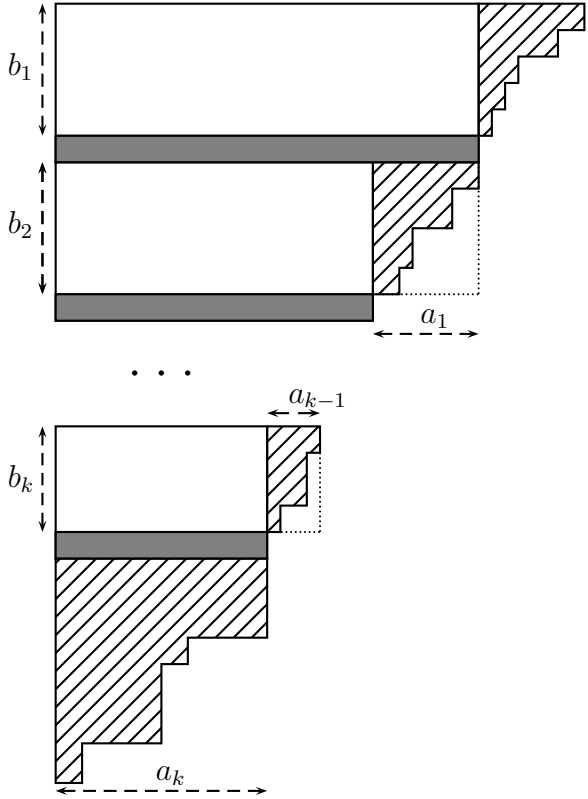

\begin{lemma}\label{L9}
For $k=1,2,\dots$ and integers
$$
1\leq x_1<x_2<\dots<x_k, \quad y_1\geq y_2\geq\dots\geq y_k\geq 0
$$
written in form \eqref{eq14}, one has
\begin{multline}\label{eq17}
P_2(x_1,\dots,x_k;y_1,\dots,y_k) \\= \langle\infty\rangle_q \frac{\langle
b_2+a_1 \rangle_q \dots \langle b_k+a_{k-1} \rangle_q }{\langle b_1
\rangle_q\dots \langle b_k \rangle_q \langle a_1 \rangle_q\dots \langle a_k
\rangle_q } ~ q^{\sum_{\{(i,j):1\leq i\leq j\leq k\}} (b_i+1)(a_j+1)}.
\end{multline}
\end{lemma}

\begin{proof}
Figure 2 shows that a diagram $\la\in\Lambda(x_1,\dots,x_k;y_1,\dots,y_k)$ is
comprised of

\begin{itemize}

\item[(i)] $k$ rectangles of size $b_m\times(a_m+\dots+a_k)$, $~m=1,\dots,k$,

\item[(ii)] $k$ rows $a_m+\dots+a_k$, $~m=1,\dots,k$,

\item[(iii)] $k-1$ diagrams enclosed in rectangles $b_{2}\times a_1, ~b_3\times
a_2,~\dots, b_k\times a_{k-1}$,

\item[(iv)] two edge diagrams, one with at most  $b_1$ rows, another with at
most $a_k$ columns.

\end{itemize}

The contribution (i) of the rectangles is the factor $q^{\sum_{\{(i,j):1\leq
i\leq j\leq k\}} b_ia_j}$. The contribution (ii) of distinguished rows is
$q^{\sum_{m=1}^k ma_m}$. It is known (see Proposition 1.3.19 in \cite{Stanley}) that
the generating function of diagrams enclosed in rectangle $b\times a$ is the
$q$-binomial coefficient
$$
\frac{\langle b+a \rangle_q}{\langle b \rangle_q \langle a \rangle_q}.
$$
From this, the contribution of (iii) is
$$
\frac{\langle b_2+a_1 \rangle_q}{\langle b_2 \rangle_q \langle a_1
\rangle_q}\cdots \frac{\langle b_k+a_{k-1} \rangle_q}{\langle b_k \rangle_q
\langle a_{k-1} \rangle_q}.
$$
Finally, the contribution of edge diagrams is
$$
\frac{1}{\langle b_1\rangle_q  \langle a_k\rangle_q}.
$$
Multiplying the factors out and recalling the normalizing factor
$\langle\infty\rangle_q$ yields the result.
\end{proof}

Now we can quickly complete the proof. The probability $P_1(x_1,\dots,x_k)$ is
easily found from the $q$-shuffle algorithm:
\begin{equation}\label{eq18}
P_1(x_1,\dots,x_k)=(1-q)^k q^{b_1+\dots+b_k}.
\end{equation}
Here we substantially used the feature that the
letters $1,\dots,k$
are the first successive letters
in the alphabet $\Z_+$.
We could not do the same with
arbitrary indices $i_1<\dots<i_k$, that is, to get a closed formula for the
probability that given generic letters $i_1,\dots,i_k$ occupy given generic
positions $x_1,\dots,x_k$ in the random word $w^+$.

Finally, substitute \eqref{eq17} and \eqref{eq18} into \eqref{eq19}, and note that
$$
a_m+b_{m-1}=d_m-d_{m-1}, \qquad m=2,\dots,k.
$$
Because of these relations, the factor
$$
\prod_{m=2}^k\langle a_m+b_{m-1} \rangle_q=\prod_{m=2}^k\langle d_m - d_{m-1} \rangle_q,
$$
is a constant and so can be taken out of the sum. This gives the desired formula.
\end{proof}

\section{Complements}\label{section7}

\subsection{A characterization of the inversion counts.}\label{subsection7.1}

In the course of proving Theorem \ref{thm3}, we established in Lemmas
\ref{lemma7} and \ref{lemma8}  a correspondence between balanced permutations
and  sequences $(r_i, ~i\in \Z)$ of their right inversion counts. By far not
every nonnegative integer sequence $(r_i, ~i\in \Z)$ can occur in this way,
thus it is of some interest to describe possible sequences in some detail.

\begin{theorem}\label{L3}
A nonnegative integer sequence $(r_i,~ i\in\Z)$ occurs as a sequence of right
inversion counts for some permutation $\sigma\in\Sym^{\rm bal}$ if and only if
the following two conditions hold:
\begin{itemize}
\item[\rm (i)] the values $(\ell_i,~ i\in\Z)$ determined from  \eqref{eq22} are
finite,
\item[\rm (ii)] $r_i=0$ for infinitely many $i\in\Z_+$.
\end{itemize}
Under these conditions such $\sigma$ is unique, and it has left inversion
counts $(\ell_i, i\in\Z)$.
\end{theorem}

\begin{proof}
Condition (i) is necessary by the definition of admissible permutation. We
prove the remaining assertions in three steps.

\medskip

{\it Step 1.} The starting point  is the following extension of the elimination
algorithm.

\medskip

\noindent {\bf The one-sided infinite elimination algorithm}: Given a
nonnegative integer sequence $(r_1, r_2, \dots)$, construct a word
$w=w_1w_2\dots$ by setting first $w_1=r_1+1$, and  for $k>1$ defining $w_k$
recursively as the $(r_k+1)$st minimal element of the reduced list
$\Z_+\setminus\{w_1,\dots, w_{k-1}\}$.
\medskip

In general, the output $w$ need not be a permutation, and we need a condition
to guarantee that.

\begin{lemma}
A nonnegative integer sequence $(r_i,~ i\in\Z_+)$ is a sequence of the right
inversion counts for some permutation $\sigma\in\Sym_{\Z_+}$ if and only if
$r_i=0$ for infinitely many $i\in\Z_+$. In this case $\sigma$ is the output of
the infinite elimination algorithm applied to $(r_i,~ i\in\Z_+)$.
\end{lemma}

\begin{proof}
Suppose $(r_i, i\in \Z_+)$ are the right inversion counts of $\sigma\in
\Sym_{\Z_+}$. For each $i\in\Z_+$ there exist $r_i+1$ positions $j\geq i$ with
$\sigma(j)\leq\sigma(i)$. Since $1\leq r_i+1<\infty$, there exists position
$j^*\geq i$ with $\sigma(j^*)=\min(\sigma(j): j\geq i)$, but then obviously
$r_{j^*}=0$. Since $i$ was arbitrary, we conclude by induction that $r_j=0$ for
infinitely many $j$.

Conversely, suppose a sequence $(r_i)$ satisfies $r_{i_k}=0$ for
$i_1<i_2<\dots$ At step $i_k$ the algorithm selects
$w_{i_k}=\min(\Z_+\setminus\{w_1,\dots,w_{i_{k-1}}\})$, hence the generic $k\in
\Z_+$ is eliminated from the original list $1\,2\,\dots$ in at most $i_k$
steps. Thus eventually every positive integer is chosen, and the output $w$ is
a permutation word encoding some permutation $\sigma\in \Sym_{\Z_+}$.

Finally, suppose the output is a permutation. The elimination algorithm implies
that exactly $r_i$ integers smaller $w_i$ remain in the list at step $i$, and
these are eventually chosen at later steps. Hence there are $r_i$ right
inversions $(i,j), j>i$ for every $i\in \Z_+$.
\end{proof}

{\it Step 2.} We turn next to the connection between permutations and strict
orders. Note that the definition of inversion counts for permutations is
applicable to orders as well. For instance, given a strict order
$\,\triangleleft\,$, the corresponding left inversion count $\ell_i$ of a
number $i$ is the cardinality of the set $\{j: j<i, \, i\,\triangleleft\,j\}$.
For $\sigma\in\Sym_{\Z_+}$ we define a strict order $\,\triangleleft\,$ on
$\Z_+$ by setting $i\,\triangleleft\, j$ iff
$\sigma(i)<\sigma(j)$\footnote{Note that here we step away from the convention
\eqref{eq20} of Section \ref{section3}, where the order was defined through
$\sigma^{-1}$.}. The elimination algorithm has an obvious modification which
derives a word $(w_j, j\in\Z_{\geq i})$ from nonnegative integer input sequence
$(r_j,j\in\Z_{\geq i})$. If $(r_i, i\in\Z_+)$ satisfies condition (ii) of
Theorem \ref{L3} then for every $i\in\Z$ the elimination algorithm transforms
$(r_j,j\in\Z_{\geq i})$ in permutation word  $(w_j,j\in\Z_{\geq i})$, which in
turn corresponds to some order $\triangleleft$ on $\Z_{\geq i}$. For two
positions $i<j$, the relation according to the order $\,\triangleleft\,$ only
depends on $r_i,\dots,r_j$, thus there is a unique order $\,\triangleleft\,$ on
$\Z$ compatible with all $(r_i, i\in\Z)$. It follows that condition (ii) holds
if and only if $(r_i, i\in\Z)$ are the right inversion counts of some order
$\,\triangleleft$.

\medskip

{\it Step 3.} Assuming order $\,\triangleleft\,$ with finite right inversion
counts $(r_i, i\in\Z)$, the left inversion counts of $\,\triangleleft\,$  are
determined recursively from \eqref{eq22}. If all $\ell_i<\infty$, then the
order $\,\triangleleft\,$ is {\it admissible}, meaning that the order has
finite left and right inversion counts.

\begin{lemma}\label{Lbij}
Let $(\ell_j),(r_i)$ be the inversion counts of admissible order on $\mathbb
Z$. The formula $\sigma(i)=i-\ell_i+r_i$ establishes a bijection between the
set of admissible orders on\/ $\Z$ and\/ $\Sym^{\rm bal}$.
\end{lemma}

\begin{proof}
Consider  admissible order $\,\triangleleft\,$ with the inversion counts
$(\ell_i),(r_i)$. The relation $i\,\triangleleft\, j$ entails that the interval
$\{k\in\Z: i\,\triangleleft\, k\,\triangleleft\, j\}$ is finite, for otherwise
at least one of the inversion counts $\ell_i$, $r_i$, $\ell_j$ $r_j$ were
infinite. For similar reason $(\Z,\triangleleft)$ has neither minimal nor
maximal element. It follows that the ordered set $(\Z,\triangleleft)$ is
isomorphic to $(\Z,<)$, but then there exists a bijection $\si:\Z\to\Z$ such
that $i\triangleleft j$ iff $\si(i)<\si(j)$. By Proposition \ref{prop2}, such
permutation $\si$ is admissible. The only freedom in the choice of $\sigma$
stems from the observation that $s^{(n)}\si$ yields the same order, where
$s^{(n)}$ is an arbitrary shift permutation. Since the shift affects the
balance we may choose $\si$ balanced, and then it is defined uniquely. Finally,
for balanced $\sigma$ the relation $\sigma(i)=i-\ell_i+r_i$ is guaranteed by
Lemma \ref{lemma7}.
\end{proof}

We have verified that conditions (i) and (ii) of Theorem \ref{L3} are also
sufficient. This completes the proof of the theorem.
\end{proof}

We did not succeed to replace condition (i) by a substantially simpler one. For
$\sigma\in\Sym^{\rm bal}$ there must be also infinitely many $i<0$ with
$r_i=0$, however this together with (ii) still does not suffice. For instance,
the  permutation defined by
$$
\sigma(-2k)=-k, \quad \sigma(-2k-1)=2k+1, \quad \sigma(k)=2k \quad \textrm{for
$k\in\Z_+$},
$$
has $r_i$ finite for all $i$, $r_i=0$ for $i\geq 0$ and even $i<0$, but
$\ell_0=\infty$.

The following result follows from Theorem \ref{thm3}. We feel nevertheless that
a direct derivation  adds insight. As before, $G$ stands for the geometric
distribution \eqref{eq23} with parameter $q\in(0,1)$.

\begin{lemma}\label{prop3}
The random collection $(r_i: i\in\Z)$ obtained by independent sampling from $G$
satisfies both conditions of Theorem \ref{L3} almost surely.
\end{lemma}

\begin{proof}
Condition (ii) is obvious, since $G$ gives positive mass to $0$.

To check (i), we  fix $j\in\Z$ and prove that for every  $r\in\{0,1,2,\dots\}$,
conditionally  given $r_j=r$, the quantity $\ell_j$ determined from
\eqref{eq22} is finite almost surely. Indeed, the sequence
$$
(x_0,x_1,x_2,\dots):=(r^{(j)}_j, r^{(j-1)}_j, r^{(j-2)}_j, \dots)
$$
defined by recursion \eqref{eq21} is a nondecreasing Markov chain on
$\{r,r+1,\dots\}$ with 0\,-1 increments, the initial state $x_0=r$, and the
one-step transition probabilities \eqref{transprob}. As we have already pointed
out, \eqref{eq22} entails that $\ell_j$ is the total number of the 0-increments
of this Markov chain. To show that $\ell_j<\infty$ almost surely it suffices to
prove that the expectation of $\ell_j$ is finite. The expectation is
$$
{\mathbb E}\ell_j=
\sum_{m=0}^\infty\sum_{k=0}^\infty\Prob(x_m=r+k)q^{r+k+1}
=q^{r+1}\sum_{k=0}^\infty q^k\sum_{m=0}^\infty \Prob(x_m=r+k),
$$
so it is enough to check that the last sum over $m\ge0$ is bounded by a constant
independent of $k$.
Note that $\Prob(x_m=r+k)$ vanishes unless $m\ge k$, so that we may start
summation from $m=k$. Then the event $x_m=r+k$ means that among the first $m$
moves of the chain there are exactly $m-k$ trivial transitions. The probability
of this event does not exceed $q^{m-k}$, because the probability of
any trivial transition is at most $q$. Therefore,
$$
\sum_{m=k}^\infty \Prob(x_m=r+k)\le \sum_{m=k}^\infty
\Prob(x_m=r+k)q^{m-k}=(1-q)^{-1},
$$
which concludes the proof.
\end{proof}

\subsection{Approximation by finite permutations.}

We proceed with the convention of Section \ref{subsection7.1}, which implies
that a bijection between $\sigma\in\Sym^{\rm bal}$ and admissible orders on $\Z$
is established via $i\,\triangleleft\, j\Leftrightarrow \sigma(i)<\sigma(j)$.
As above, for an interval $I\subset\Z$, we denote by $\Sym_I\subset\Sym$ the
subgroup of permutations that do not move integers outside $I$. Below we will
deal with finite intervals only; then $\Sym_I$ is naturally isomorphic to the
finite symmetric group of degree $\#I$. For $\sigma\in\Sym^{\rm bal}$  we have
a unique increasing bijection between $\{\sigma(i), i\in I\}$ and $I$.
Replacing each $\sigma(i)$ with its counterpart  by this bijection, and
otherwise setting $\sigma_I(i)=i$ for $i\notin I$ maps $\sigma\in\Sym^{\rm
bal}$ to some $\sigma_I\in\Sym_I$,  a {\it truncation\/} of $\si$. In other
words, the orders induced by $\si$ and $\si_I$ coincide on $I$. For instance,
for $I=\{1,2,3\}$, the truncation of any permutation with pattern
$\dots|4~3~2~1~6\dots$ gives $\dots0|3~2~1~4\dots$ (here the vertical bar is used
to separate positions $0$ and $1$).

With reference to the discussion at the end  of Section \ref{section2}, we
endow  $\Sym^{\rm bal}$ with the topology inherited from $\Sym$. Recall that in
this topology, the convergence $\sigma_n\to\sigma$ means
$\sigma_n(i)=\sigma(i)$ for each $i$ and all $n$ larger than some $n(i)$. In
what follows $I_n$ $(n\in{\mathbb N})$ stands for arbitrary increasing sequence
of finite intervals in $\Z$ whose union is the whole $\Z$. One obvious choice
is $I_n=\{-n,-n+1,\dots,n-1,n\}$.

\begin{lemma}
For every  $\sigma\in\Sym^{\rm bal}$ we have $\sigma_{I_n}\to \sigma$.
\end{lemma}

\begin{proof}
Let $(\ell_i)$ and $(r_i)$ be the inversion counts of $\si$, and let
$(\ell^{(n)}_i)$ and $(r^{(n)}_i)$ be the similar quantities for $\si_{I_n}$.
For every fixed $i\in\Z$, the sequences $(\ell_i^{(n)})$ and $(r_i^{(n)})$
are nondecreasing. Indeed, the larger set $I_n$ the larger
the set of inversions in
$\sigma_{I_n}$
associated with $i$. Next,
observe that these sequences stabilize for large $n$ to values $\ell_i$ and
$r_i$, respectively. Since $\si(i)=i-\ell_i+r_i$ and
$\si_{I_n}(i)=i-\ell^{(n)}_i+r^{(n)}_i$, this just means that $\si_{I_n}\to\si$.
\end{proof}

Immediately from the lemma and considerations in Section \ref{section3} we
derive:

\begin{proposition}\label{propopo}
Let\/ $\Sigma$ be the random  Mallows permutation
of\/ $\Z$, and let \/ $\Sigma_{I_n}$ be the truncation of\/ $\Si$ corresponding to
$I_n$. Then $\Sigma_{I_n}\to\Sigma$ with probability one.
\end{proposition}

Note that the law of $\Sigma_{I_n}$ is essentially the Mallows distribution on
the finite symmetric group of degree $\#I_n$.

Proposition \ref{propopo} enables us to give alternative derivations to
Corollary \ref{cor1} and to the most technical part of Theorem \ref{thm2}
(Lemma \ref{lemma2}) from the analogous properties of the Mallows distributions
\eqref{MalFin} on finite symmetric groups. To that end, one just needs to
observe that the inversion mapping $\sigma\to\sigma^{-1}$ is a homeomorphism of
$\Sym^\bal$ endowed with the weak topology. Now, $\Sigma_{I_n}\to\Sigma$ taken
together with the equality in the law
$\Sigma_{I_n}\stackrel{d}{=}(\Sigma_{I_n})^{-1}$ for finite permutations imply
$\Sigma\stackrel{d}{=}\Sigma^{-1}$.

We emphasize that it is nowhere stated, nor is true, that
$(\sigma_I)^{-1}=(\sigma^{-1})_I$. Examples are easily designed to show that
the truncation and the inversion operation do not commute.

Finally, we note that the Mallows measures on $\Sym_I$ with finite $I$, and the
system of the joint distributions of $(\Sigma(i), ~i\in I)$ (computed in
Section \ref{section6}) are two competing families of ``finite dimensional
distributions'' for $\Sigma$ considered as a random function. The suitability
of one or another system in concrete situation depends on the nature of the
statistic of permutation under study.

\end{document}